\documentclass[12pt,a4paper,leqno]{article}

\usepackage{latexsym}
\usepackage{amssymb,exscale}
\usepackage[centertags]{amsmath}
\usepackage{amsthm}
\usepackage{graphicx}

\usepackage{hyperref}

\theoremstyle{definition}

\swapnumbers

\newtheorem*{theorem*}{Theorem}
\newtheorem{proposition}[equation]{Proposition}
\newtheorem{lemma}[equation]{Lemma}

\newtheorem{definition}[equation]{Definition}
\newtheorem{remark}[equation]{Remark}

\renewcommand{\(}{\bigl(}
\renewcommand{\)}{\bigr)\vphantom{)}}

\newcommand{\imply}{\;\;\;\Longrightarrow\;\;\;}

\newcommand{\eps}{\varepsilon}

\newcommand{\F}{\mathcal F}

\newcommand{\R}{\mathbb R}

\renewcommand{\Pr}[1]{\mathbb{P}\mskip1.5mu\(\mskip1.5mu#1\mskip1.5mu\)}
\newcommand{\cP}[2]{\mathbb{P}\mskip1.5mu\(\mskip1.5mu#1\mskip1.5mu
 \big|\mskip1.5mu#2\mskip1.5mu\)}

\DeclareMathOperator{\dist}{dist}

\newcommand{\Unco}{\operatorname{Unco}}
\newcommand{\Proj}{\operatorname{Proj}}

\newcommand{\sif}{$\sigma$\nobreakdash-field}

\newcommand{\pixelated}[1]{$#1$\nobreakdash-\hspace{0pt}pixelated}

\begin{document}

\title{Random compact set  meets the graph\\
 of nonrandom continuous function}

\author{Boris Tsirelson}

\date{}
\maketitle

\begin{abstract}
On the plane, every random compact set with almost surely uncountable
first projection intersects with a high probability the graph of some
continuous function.

Implication: every black noise over the plane fails to factorize when
the plane is split by such graph.
\end{abstract}

\setcounter{tocdepth}{1}
\tableofcontents

\section*{Introduction and result}
\addcontentsline{toc}{section}{Introduction and result}
The following question is motivated by the theory of black noises (see
Sect.~\ref{sect3}).

Let $K$ be a random compact subset of the square $ [0,1]\times[0,1]
$. Given a continuous function $ f : [0,1] \to [0,1] $ we consider
its graph $ G_f = \{ (x,f(x)) : 0 \le x \le 1 \} $, the probability $
\Pr{ K \cap G_f \ne \emptyset } $ that $K$ meets the graph, and its
supremum
\[
\sup_f \Pr{ K \cap G_f \ne \emptyset }
\]
over all continuous functions $f$. The question is, whether the
supremum is equal to $1$, or not.

Clearly, this probability need not be positive if $K$ is (at most)
countable almost surely. The same holds if $K$ is uncountable but its
first projection
\[
\Proj_1 (K) = \{ x \in [0,1] : \exists y \in [0,1] \;\> (x,y) \in K \}
\]
is (at most) countable.

We denote by $ \Unco_1 $ the set of all uncountable subsets of
$[0,1]$, and by $ \Unco_2 $ the set of all subsets of
$[0,1]\times[0,1]$ whose first projection is uncountable.

\begin{theorem*}
Let a random compact set $ K \subset [0,1]\times[0,1] $ belong to $
\Unco_2 $ almost surely. Then for every $ \eps>0 $ there exists a
continuous function $ f : [0,1] \to [0,1] $ such that $ \Pr{ K \cap
G_f \ne \emptyset } \ge 1-\eps $.
\end{theorem*}

In particular, $K$ can be a set of zero Hausdorff dimension, shifted
at random. The corresponding $f$ cannot be nice!

\smallskip

The author thanks Ohad Feldheim and Tom Ellis for helpful discussions
related to Sect.~\ref{sect3}.

\numberwithin{equation}{section}

\section[The construction]
  {\raggedright The construction}
\label{sect1}
We use the probabilistic method: construct a random $f$ and prove that
it fits with positive probability.

\begin{definition}
A \emph{\pixelated{(m,n)} graph} is a connected subset of $
[0,1]\times[0,1] $ of the form
\[
G = \bigcup_{k=1}^{2^m} \Big[ \frac{k-1}{2^m},\frac{k}{2^m} \Big]
\times \Big[ \frac{l_k-1}{2^n},\frac{l_k}{2^n} \Big]
\]
where $ l_1,\dots,l_{2^m} \in \{1,\dots,2^n\} $.
\end{definition}

\[
\begin{gathered}\includegraphics{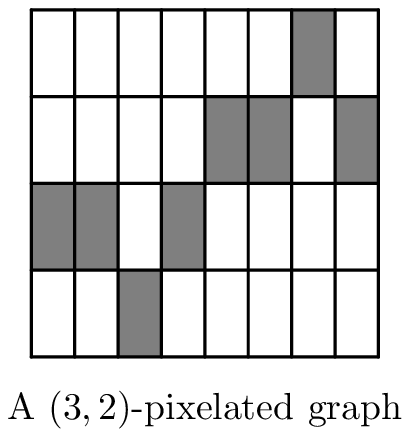}\end{gathered}
\]

\begin{lemma}\label{1.2}
Assume that $ m_1<m_2<\dots $, and \pixelated{(m_n,n)} graphs $ G_n $
are a decreasing sequence: $ G_1 \supset G_2 \supset G_3 \supset \dots
$ Then their intersection $ \cap_n G_n $ is the graph of some
continuous function $ [0,1] \to [0,1] $.
\end{lemma}

The proof is left to the reader.

\[
\begin{gathered}\includegraphics{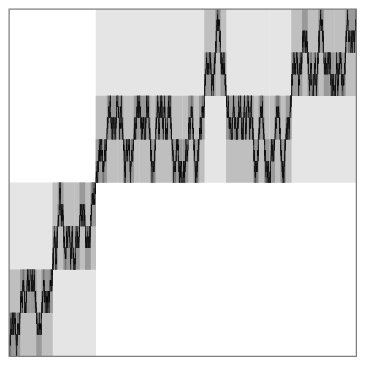}\end{gathered}
\]

Given $ m_1<m_2<\dots $ we construct a random decreasing sequence $
G_1 \supset G_2 \supset G_3 \supset \dots $ of \pixelated{(m_n,n)}
graphs $ G_n $ (and therefore, by Lemma \ref{1.2}, a random continuous
function) as follows.

First, the set of all \pixelated{(m,n)} graphs contains $ (2^n)^{2^m}
= 2^{n\cdot2^m} $ elements. In particular, the set of all
\pixelated{(m_1,1)} graphs contains $ 2^{2^{m_1}} $ elements, indexed
by $ 2^{m_1} $ bits $ l_1,\dots,l_{2^{m_1}} \in \{1,2\} $. We
choose at random one of these graphs (each with probability $
2^{-2^{m_1}} $); this is $ G_1 $.

Conditionally, given $ G_1 $, we consider the set of all
\pixelated{(m_2,2)} graphs contained in $ G_1 $. Each of these
consists of $ 2^{m_1} $ rescaled \pixelated{(m_2-m_1,1)} graphs
(situated in the $ 2^{m_1} $ rectangles that constitute $ G_1
$). Thus, the number of elements in this set is at most $ (
2^{2^{m_2-m_1}} )^{2^{m_1}} = 2^{2^{m_2}} $. It is in fact less than $
2^{2^{m_2}} $ because of connectedness. However, connectedness
restricts at most two bits (the leftmost and the rightmost) in each
block of $ 2^{m_2-m_1} $ bits. Thus, the number of elements is at
least $ ( 2^{2^{m_2-m_1}-2} \)^{2^{m_1}} = 2^{2^{m_2}-2^{m_1+1}} $. We
fix the two bits (leftmost and rightmost) in each block as to ensure
connectedness (no matter how do we fix them when we have a choice) and
randomize all other bits; this way we define $ G_2 $.

And so on.

\section[The proof]
  {\raggedright The proof}
\label{sect2}
\begin{proposition}\label{2.1}
Let a random compact set $ K \subset [0,1]\times[0,1] $ belong to $
\Unco_2 $ almost surely. Then for every $ \eps>0 $ there exist $ m_1 <
m_2 < \dots $ such that the random decreasing sequence $ G_1 \supset
G_2 \supset \dots $ constructed in Sect.~\ref{sect1} satisfies
\[
\Pr{ K \cap G_n \ne \emptyset } \ge 1-\eps
\]
for all $n$. (It is meant that the two random objects $K$ and
$(G_n)_n$ are independent.)
\end{proposition}

Before proving the proposition we check that it implies the
theorem. First,
\[
\Pr{ \forall n \;\> K \cap G_n \ne \emptyset } \ge 1-\eps
\]
by monotonicity. Second,
\[
K \cap G_n \ne \emptyset \imply \dist(K,G) \le 2^{-n}
\]
where $ G = \cap_n G_n $. By Lemma \ref{1.2}, $ G = G_f $ for some
random continuous function $f$. By compactness,
\[
\Pr{ K \cap G_f \ne \emptyset } = \Pr{ \dist(K,G_f)=0 } \ge 1-\eps \,
.
\]
Therefore the conditional probability $ \cP{ K \cap G_f \ne \emptyset
}{ f } $ is at least $1-\eps$ with a positive probability, and we are
done.

It remains to prove Prop.~\ref{2.1}.

\begin{lemma}\label{2.2}
If $ B \in \Unco_1 $ then
\[
\# \Big\{ k \in \{1,\dots,2^n\} : B \cap \Big( \frac{k-1}{2^n},
\frac{k}{2^n} \Big) \in \Unco_1 \Big\} \uparrow \infty \quad \text{as
} n \to \infty \, .
\]
\end{lemma}

\begin{proof}
Otherwise we get a finite subset of $ [0,1] $ such that $ B $ is
locally countable outside this finite set.
\end{proof}

Recall that, according to Sect.~\ref{sect1}, $ G_1 $ is a random
\pixelated{(m_1,1)} graph. In the next lemma we denote it by $
G_1(m_1) $.

\begin{lemma}\label{2.3}
If $ A \in \Unco_2 $ then
\[
\Pr{ A \cap G_1(m_1) \in \Unco_2 } \to 1 \quad \text{as } m_1 \to
\infty \, .
\]
\end{lemma}

\begin{proof}
Lemma \ref{2.2} applied to the first projection $ B = \Proj_1(A) \in
\Unco_1 $ of $ A $ gives $ i(m_1) \uparrow \infty $ as $ m_1 \to
\infty $, where
\[
i(m_1) = \# \Big\{ k \in \{1,\dots,2^{m_1}\} : B \cap \Big(
\frac{k-1}{2^{m_1}}, \frac{k}{2^{m_1}} \Big) \in \Unco_1 \Big\} \, .
\]
On the other hand,
\[
\Pr{ A \cap G_1(m_1) \in \Unco_2 } \ge 1 - 2^{-i(m_1)} \, ,
\]
since the relation $ B \cap \( \frac{k-1}{2^{m_1}},\frac{k}{2^{m_1}}
\) \in \Unco_1 $ implies at least one of two relations $ A \cap \(
[ \frac{k-1}{2^{m_1}},\frac{k}{2^{m_1}} ] \times [
  \frac{l_k-1}{2},\frac{l_k}{2} ] \) \in \Unco_2 $, $ l_k=1,2 $.
\end{proof}

\begin{remark}\label{2.4}
Lemma \ref{2.3} remains true if the leftmost and rightmost bits $ l_1,
l_{2^{m_1}} $ are fixed (not randomized). In the proof we just replace
$ 2^{-i(m_1)} $ with $ 2^{-(i(m_1)-2)} $.
\end{remark}

\begin{proof}[Proof of Prop.~\ref{2.1}]
It is sufficient to ensure $ \Pr{ K \cap G_1 \notin \Unco_2 } \le
\frac\eps2 $ and
\[
\Pr{\, K \cap G_n \in \Unco_2 \;\land\; K \cap G_{n+1} \notin
\Unco_2 \,} \le \frac\eps{2^{n+1}}
\]
for all $n$.

First, by Lemma \ref{2.3}, the conditional probability
\[
\cP{ K \cap G_1(m_1) \in \Unco_2 }{ K } \to 1 \quad \text{as
} m_1 \to \infty
\]
almost surely. The similar relation for unconditional probability
follows. Thus, $ \Pr{ K \cap G_1 \notin \Unco_2 } \le \frac\eps2 $ if
$ m_1 $ is large enough.

Second, we prove that
\[
\Pr{\, K \cap G_1 \in \Unco_2 \;\land\; K \cap G_2 \notin \Unco_2 \,}
\to 0 \quad \text{as } m_2 \to \infty
\]
(and therefore does not exceed $ \frac\eps4 $ if $m_2$ is large
enough). To this end it is sufficient to prove that $ \cP{ K \cap G_2
\notin \Unco_2 }{ K,G_1 }  \to 0 $ for almost all pairs $ (K,G_1) $
such that $ K \cap G_1 \in \Unco_2 $.

As noted in Sect.~\ref{sect1}, conditionally, given $ G_1 $, the
\pixelated{(m_2,2)} graph $ G_2 $ consists of $ 2^{m_1} $ rescaled
\pixelated{(m_2-m_1,1)} graphs situated in the $ 2^{m_1} $ rectangles
that constitute $ G_1 $. At least one of these rectangles $ R $
satisfies $ K \cap R \in \Unco_2 $. It remains to apply Lemma
\ref{2.3} (and Remark \ref{2.4}) to $ G_2 \cap R $ (rescaled).

The same argument works for $ G_3 $, $ G_4 $ and so on.
\end{proof}

\section[Implications for noise theory]
  {\raggedright Implications for noise theory}
\label{sect3}
We still have no definitive presentation of the general theory of
noises over the plane, but we have examples \cite{SS}, \cite{EF} and
some discussion \cite[Sect.~1f]{Ts11}. For noises over the line we
have a general theory \cite{Ts04}, including the spectral theory
\cite[Sect.~9]{Ts04} (see also \cite[Sect.~7b]{Ts11}).

For a black noise over $\R$ spectral sets are uncountable compact
subsets of $\R$ (except for the empty set). Likewise, for a black
noise over $\R^2$ spectral sets are uncountable compact subsets of
$\R^2$. Moreover, they belong to $ \Unco_2 $, since the corresponding
noise over $\R$ is still black \cite[Sect.~1f]{Ts11}. Thus, they meet
the graph of some continuous function $ f : \R \to \R $ with a
positive probability. This means failure of the factorization
property:
\[
\F_{G_f^-} \vee \F_{G_f^+} \ne \F \, ;
\]
here $ G_f^- = \{ (x,y) : y<f(x) \} $, $ G_f^+ = \{ (x,y) : y>f(x) \}
$ are planar domains, and $ \F_{G_f^-} $, $ \F_{G_f^+} $ are the
corresponding sub-\sif s of the \sif\ $ \F = \F_{\R^2} $ generated by
the black noise.

It is interesting to consider ``curvilinear strips''
\[
G_f^{a,b} = \{ (x,y) : a+f(x) < y < b+f(x) \} \subset \R^2
\]
indexed by intervals $ (a,b) \subset \R $, and the corresponding
sub-\sif s
\[
\F_{a,b} = \F_{G_f^{a,b}} \subset \F \, .
\]
We observe that
\[
\F_{a,b} \text{ and } \F_{b,c} \text{ are independent}
\]
whenever $ a<b<c $; also,
\[
\F_{a,b+\eps} \vee \F_{b,c} = \F_{a,c}
\]
whenever $ a<b<b+\eps<c $; and nevertheless
\[
\F_{a,b} \vee \F_{b,c} \ne \F_{a,c} \, .
\]

\bigskip
\filbreak
{
\small
\begin{sc}
\parindent=0pt\baselineskip=12pt
\parbox{4in}{
Boris Tsirelson\\
School of Mathematics\\
Tel Aviv University\\
Tel Aviv 69978, Israel
\smallskip
\par\quad\href{mailto:tsirel@post.tau.ac.il}{\tt
 mailto:tsirel@post.tau.ac.il}
\par\quad\href{http://www.tau.ac.il/~tsirel/}{\tt
 http://www.tau.ac.il/\textasciitilde tsirel/}
}

\end{sc}
}
\filbreak


\begin{thebibliography}{8.}
\addcontentsline{toc}{section}{References}

{\raggedright
\bibitem{SS}
\textsc{Schramm, O., Smirnov, S.} with an appendix by \textsc{Garban,
 C.} (2011).
On the scaling limits of planar percolation.
\textit{Ann. Probab.} \textbf{39}:5 1768--1814.

\bibitem{EF}
\textsc{Ellis, T.} and \textsc{Feldheim, O. N.} (2012).
The Brownian web is a two-dimensional black noise.
\texttt{arXiv:1203.3585}.

\bibitem{Ts11}
\textsc{Tsirelson, B.} (2011).
Noise as a Boolean algebra of sigma-fields.
\texttt{arXiv:1111.7270}.

\bibitem{Ts04}
\textsc{Tsirelson, B.} (2004).
Nonclassical stochastic flows and continuous products.
\textit{Probability Surveys} \textbf{1} 173--298.

}
\end{thebibliography}
\end{document}